\documentclass[11pt]{amsart}
\usepackage{fullpage, amssymb, amscd}

\newtheorem{theorem}{Theorem}
\newtheorem{proposition}[theorem]{Proposition}

\newtheorem{corollary}[theorem]{Corollary}

\theoremstyle{definition}

\newtheorem{example}[theorem]{Example}

\theoremstyle{remark}
\newtheorem{remark}[theorem]{Remark}

\numberwithin{equation}{section}
\numberwithin{theorem}{section}

\newcommand{\CC}{\mathbb{C}}

\newcommand{\Mn}{\mathbb{M}_n}

\newcommand{\rank}{\textnormal{rank}}
\newcommand{\range}{\textnormal{range}}

\newcommand{\tr}{\textnormal{tr}}

\begin{document}

\title[$AB$ versus $BA$]{On the similarity of $AB$ and $BA$ for \\ normal and other matrices}

\author{Stephan Ramon Garcia}
	\address{Department of Mathematics, Pomona College, 610 N. College Ave., Claremont, CA 91711} 
	\email{stephan.garcia@pomona.edu}
	\urladdr{http://pages.pomona.edu/~sg064747}
	
	\author{David Sherman}
	\address{Department of Mathematics, University of Virginia, P.O. Box 400137, Charlottesville, VA 22904-4137}
	\email{dsherman@virginia.edu}
	\urladdr{http://people.virginia.edu/~des5e}
	
	\author{Gary Weiss}
	\address{Department of Mathematical Sciences, 4199 French Hall West, University of Cincinnati, 2815 Commons Way, Cincinnati, OH 45221-0025}
	\email{gary.weiss@uc.edu}
	\urladdr{http://math.uc.edu/~weiss}

\subjclass[2010]{15A03, 15A18}
\keywords{matrix similarity}
\date{\today}
\thanks{The authors would like to acknowledge the support of NSF grants DMS-1265973 and DMS-1201454, and of Simons Collaboration Grant for Mathematicians 245014.}

\begin{abstract}
It is known that $AB$ and $BA$ are similar when $A$ and $B$ are Hermitian matrices.  In this note we answer a question of F. Zhang by demonstrating that similarity can fail if $A$ is Hermitian and $B$ is normal.  Perhaps surprisingly, similarity does hold when $A$ is positive semidefinite and $B$ is normal.
\end{abstract}

\maketitle

\section{Introduction}
Throughout this paper $A$ and $B$ denote complex square matrices of the same size.  We pursue the following question: when is $AB$ similar to $BA$?

This does not always happen.  But it does when $A$ and $B$ are Hermitian or when either is invertible; we seek other assumptions that imply similarity.  For instance, it was asked by F. Zhang (personal communication) whether it suffices for $A$ and $B$ to be merely normal.  We show here that similarity does not follow even when $A$ is Hermitian and $B$ is normal (Example \ref{T:ctrex}), although it does if $A$ is further assumed to be positive semidefinite (Theorem \ref{T:pos}).  We also show that similarity, or unitary similarity, follows under various hypotheses when one or both matrices have low rank or size, and we give minimal counterexamples showing that our conditions are sharp.

Similarity will be denoted by $\sim$ and unitary similarity by $\sim_u$.

We thank Fuzhen Zhang for bringing this problem to our attention, and Roger Horn for suggesting significant improvements to Section 6.

\section{Ranks of powers of a matrix}

We define the \textit{rank sequence} of $A$ to be $\{\rank (A^j)\}_{j=0}^\infty$ (with $A^0 = I$).  Which sequences of nonnegative integers occur as the rank sequence of a matrix?

Since rank is unchanged by similarity, we may as well consider the Jordan form of $A$.  Jordan blocks for nonzero eigenvalues are invertible, and the ranks of powers of a Jordan block for a zero eigenvalue drop by one until reaching zero.  So the drop from $\rank (A^j)$ to $\rank (A^{j+1})$ is precisely the number of Jordan blocks for zero of size at least $j+1$.  The size of these drops is then nonincreasing in $j$, leading to the conclusion that rank sequences are nonincreasing and convex.  We use this fact in Section 5.  (Actually it is a characterization of rank sequences, as any nonincreasing convex sequence of nonnegative integers is the rank sequence of a matrix whose Jordan blocks satisfy the criterion just mentioned.  Details are left to the interested reader.)

The rank sequence of $A$ carries the same information as the Jordan structure of $A$ for the zero eigenvalue, which is more commonly encoded in the Segre or Weyr characteristic (see \cite{LS}), but rank sequences are more natural for this paper.

\section{Known facts}

If one of the matrices is invertible, then $AB \sim BA$ (conjugate by the invertible one).  But even for $2 \times 2$ matrices, $AB$ need not be similar to $BA$:  consider
$$A = \left[ \begin{matrix}  0 & 1 \\ 0 & 0 \end{matrix} \right], \qquad B = \left[ \begin{matrix}  0 & 0 \\ 0 & 1 \end{matrix} \right].$$

It is known that for square matrices in general, the invertible Jordan blocks of $AB$ and $BA$ are the same (\cite[Theorem 3.2.11.1]{HJ}, see also \cite{F} for comparison of the Jordan structures of $AB$ and $BA$ at 0).  As a consequence we have

\begin{proposition} \label{T:rankseq}
${}$
\begin{enumerate}
\item[(i)] The rank sequences of $AB$ and $BA$ eventually become the same constant (the sum of the ranks of their invertible Jordan blocks).
\item[(ii)] $AB$ and $BA$ are similar if and only if they have the same rank sequences.
\end{enumerate}
\end{proposition}

Here are some other useful known facts.

\begin{proposition} \label{T:rank}
${}$
\begin{enumerate}
\item[(i)] If $\rank (AB) = \rank (BA) = \rank (A)$, then $AB \sim BA$.
\item[(ii)] If $A$ and $B$ are normal, then $\rank(AB) = \rank(BA)$.
\item[(iii)] If $A$ and $B$ are Hermitian, then $AB \sim BA$.
\end{enumerate}
\end{proposition}

A proof of (i) is explained in \cite[Exercise 3.2.P20b]{HJ}.  Here are short proofs of (ii) and (iii).  Using normality and the fact that $\rank(T^*T) = \rank(T) = \rank(T^*)$ for any matrix $T$,
\begin{align*} \rank (AB) &= \rank(B^*A^* A B) = \rank(B^*A A^* B) = \rank (A^*B)  = \rank(B^*A) \\ &= \rank (A^* B B^* A) = \rank(A^* B^* B A) = \rank(BA).
\end{align*}
When $A$ and $B$ are Hermitian, we note that $\rank((AB)^j) = \rank(((AB)^j)^*) = \rank((BA)^j)$ for all $j$, then apply Proposition \ref{T:rankseq}(ii).

Actually there is a sort of converse to (iii): a matrix is similar to its adjoint if and only if it is a product of two Hermitian matrices \cite[Theorem 4.1.7]{HJ}.

\begin{remark} \label{T:inf}
Proposition \ref{T:rank}(iii), which motivates the main questions in this paper, is not true for infinite-dimensional Hilbert space operators.  Let $A$ be the diagonal operator on $\ell^2$ whose diagonal is $1, \frac12, \frac13, \dots$, and let $B$ be the projection onto the orthogonal complement of the $\ell^2$ vector $v = (1, \frac12, \frac13, \dots )$.  Then $BA$ is injective since $v$ is not in the range of $A$, but $AB$ has nontrivial kernel, namely $\CC v$.  Thus $AB$ and $BA$ cannot be similar.

% Let $S$ be the unilateral shift on $\ell^2$, and set $A = SS^*$ and $B = S+S^*$.  Then $BA = (I + S^2) S^*$, which has nontrivial kernel since $S^*$ does.  So if $AB \sim BA$, then $AB = S(I + S^{*2})$ would also have nontrivial kernel.  Because $S$ is injective, $I + S^{*2}$ would have nontrivial kernel, but it is easy to check that $-1$ is not an eigenvalue for $S^{*2}$ (in fact $S^{*j} v \to 0$ for any vector $v$).
\end{remark} 

\section{Unitary similarity}

The reader may wonder about unitary similiarity.

It may not be true that $AB \sim_u BA$ when $A$ and $B$ are Hermitian: take
$$A = \left[ \begin{matrix}  1 & 0 & 0 \\ 0 & 1 & 0 \\ 0 & 0 & 0 \end{matrix} \right], \qquad B=  i \left[ \begin{matrix}  0 & -1 & 1 \\ 1 & 0 & -1 \\ -1 & 1 & 0 \end{matrix} \right].$$  (To verify that the products are not unitarily similar, one can check that $X^* X^2 (X^*)^2 X$ has different traces for $X = AB$ and $X = BA$.)  This is a counterexample of minimal size and rank, as we now show.

\begin{proposition} \label{T:unit}
Let $A,B \in \Mn$ be normal.  Then $AB \sim_u BA$ when (i) $n\leq2$ or (ii) $\rank(A)\leq 1$.
\end{proposition}

\begin{proof}
We discuss only the nontrivial cases $n=2$ and $\rank(A)=1$.

(i) The triple $(\tr (X), \tr (X^2), \tr (X^*X))$ is a complete unitary invariant for $2 \times 2$ matrices \cite{M}.  We use the trace property for $\tr(AB) = \tr(BA)$ and
$$\tr((AB)^2) = \tr(ABAB) = \tr(BABA) = \tr((BA)^2),$$
then mix in normality to obtain
$$\tr((AB)^*(AB)) = \tr(B^*A^*AB) = \tr(A^*ABB^*) = \tr(AA^*B^*B) = \tr(A^*B^*BA) = \tr((BA)^*(BA)).$$

(ii) A rank one normal matrix is a scalar multiple of a rank-one projection, so after scaling we may find a unit vector $v$ such that $A = vv^*$, the projection onto $\CC v$.  By normality we have $\|Bv\| = \|B^*v\| = c$.  If $c=0$ then
$$AB=vv^*B = v(B^*v)^* = v(0)= 0= (0) v^* = Bvv^*= BA$$ and we are done.

If $c \neq 0$, then there is an isometry from $\text{span}\{v, c^{-1} B^*v\}$ onto $\text{span}\{c^{-1}Bv, v\}$ determined by sending $v$ to $c^{-1}Bv$ and $c^{-1} B^*v$ to $v$; these are all unit vectors, and the inner products of the pairs agree:
$$(c^{-1} B^*v)^* v = v^* (c^{-1} Bv).$$
Let $U$ be any unitary matrix that extends this isometry.  Then
$$BAU = Bvv^*U = Bv(U^*v)^* = Bv(c^{-1}B^*v)^* = (c^{-1}Bv)v^*B = (Uv)v^*B = UAB. \qedhere$$
\end{proof}

\begin{remark}
There is a similar pattern for transposes.  It is known that for any matrix $A$, $A \sim A^T$ \cite[Theorem 3.2.3.1]{HJ}.  For $2 \times 2$ matrices and rank 1 matrices we even have unitary similarity by arguments similar to Proposition \ref{T:unit}, and these conditions are sharp: consider 
$$A = \left[ \begin{matrix}  0 & 1 & 0 \\ 0 & 0 & 2 \\ 0 & 0 & 0 \end{matrix} \right].$$
For more on the condition $A \sim_u A^T$, see \cite{GT}.
\end{remark}

\section{Similarity of products of normals}

\begin{proposition}
Let $A,B \in \Mn$ be normal.  Then $AB \sim BA$ if $\rank(A)\leq 2$.
\end{proposition}

\begin{proof}
If $\rank(A) \leq 1$, we are done by Proposition \ref{T:unit}(ii), so assume that $\rank(A)=2$.  

Recall from Section 2 that rank sequences are nonincreasing and convex.  Paired with the constraint that $\rank(A)=2$, this leaves only seven possibilities for the rank sequences for $AB$ and $BA$ (although some are impossible for $n=1,2$ or $3$):
\begin{itemize}
\item $n,0, \dots$
\item $n,1,0,\dots$
\item $n,1,1, \dots$
\item $n,2,0,\dots$
\item $n,2,1,0,\dots$
\item $n,2,1,1,\dots$
\item $n,2,2,\dots$
\end{itemize}
The rank sequences for $AB$ and $BA$ have the same second entry by Proposition \ref{T:rank}(ii).  If it is 2, we have similarity by Proposition \ref{T:rank}(i).  If it is 1, then Proposition \ref{T:rankseq}(i) forces the rank sequences to be the same one out of the two possibilities above; if it is 0, there is only one possible rank sequence -- in either case we have similarity by Proposition \ref{T:rankseq}(ii).
\end{proof}

\begin{corollary}
If $A,B \in \mathbb{M}_3$ are normal, then $AB \sim BA$.
\end{corollary}

\begin{proof}
If neither is invertible, both have rank $\leq 2$.
\end{proof}

Thus a minimal counterexample for similarity of product pairs of two normal matrices would be $4 \times 4$ matrices $A$ and $B$ of rank 3.  The rank sequences of $AB$ and $BA$ should be different (Proposition \ref{T:rankseq}(ii)) but must have the same two first terms (Proposition \ref{T:rank}(ii)) and the same limit (Proposition \ref{T:rankseq}(i)).  By Proposition \ref{T:rank}(i) the rank of $AB$ cannot be 3, so the rank sequences are in the list above, and the only possibility is for them to be the fourth and fifth ones.  Such matrices exist!

\begin{example} \label{T:ctrex}
Let
$$A = \left[ \begin{matrix}  0 & 0 & 0 & 1 \\ 0 & 1 & 0 & 0 \\ 0 & 0 & 0 & 0 \\ 1 & 0 & 0 & 0 \end{matrix} \right], \qquad B = \left[ \begin{matrix}  0 & 0 & 0 & 0 \\ 0 & 0 & 0 & 1 \\ 0 & 1 & 0 & 0 \\ 0 & 0 & 1 & 0 \end{matrix} \right].$$  Then 
$$AB =  \left[ \begin{matrix}  0 & 0 & 1 & 0 \\ 0 & 0 & 0 & 1 \\ 0 & 0 & 0 & 0 \\ 0 & 0 & 0 & 0 \end{matrix} \right], \qquad BA = \left[ \begin{matrix}  0 & 0 & 0 & 0 \\ 1 & 0 & 0 & 0 \\ 0 & 1 & 0 & 0 \\ 0 & 0 & 0 & 0 \end{matrix} \right],$$
which satisfy $(AB)^2 = 0 \neq (BA)^2$ and so are not similar.  In fact their rank sequences are the fourth and fifth in the list above.  Thus Zhang's question, as stated in the introduction, has a negative answer using 0-1 matrices and $A$ even Hermitian.  
\end{example}

We conclude this section by exhibiting another class of normal matrices, other than the Hermitians, for which $AB \sim BA$.  For any square matrix $X$, define
$$\Phi(X) = \left[ \begin{matrix}  X & X^* \\ X^* & X \end{matrix} \right].$$

\begin{proposition}
Let $X,Y \in \Mn$.  Then $\Phi(X), \Phi(Y) \in \mathbb{M}_{2n}$ are normal matrices satisfying $\Phi(X)\Phi(Y) \sim \Phi(Y)\Phi(X)$.
\end{proposition}

\begin{proof}
Normality of $\Phi(X)$ and $\Phi(Y)$ is a straightforward computation.

Write $X = X_1 + i X_2$ and $Y = Y_1 + iY_2$, where $X_1,X_2,Y_1,Y_2$ are Hermitian.  Let $U \in \mathbb{M}_{2n}$ be the unitary matrix  $$\frac{1}{\sqrt{2}} \left[ \begin{matrix} I_n &  I_n \\ -I_n & I_n \end{matrix} \right].$$  Then
\begin{align*} \Phi(X)\Phi(Y) &\sim_u (U \Phi(X) U^{-1})(U \Phi(Y) U^{-1}) = \left[ \begin{matrix}  2X_1 & 0 \\ 0 & 2iX_2 \end{matrix} \right] \left[ \begin{matrix}  2Y_1 & 0 \\ 0 & 2iY_2 \end{matrix} \right] \\ &= \left[ \begin{matrix}  4X_1 Y_1 & 0 \\ 0 & -4X_2 Y_2 \end{matrix} \right] \sim   \left[ \begin{matrix}  4Y_1 X_1 & 0 \\ 0 & -4Y_2 X_2 \end{matrix} \right] \\ &= \left[ \begin{matrix}  2Y_1 & 0 \\ 0 & 2iY_2 \end{matrix} \right] \left[ \begin{matrix}  2X_1 & 0 \\ 0 & 2iX_2  \end{matrix} \right] = (U \Phi(Y) U^{-1})(U \Phi(X) U^{-1})   \sim_u \Phi(Y)\Phi(X),
\end{align*}
where the middle similarity is the direct sum of similarities obtained by Proposition \ref{T:rank}(iii).
\end{proof} 

\section{Positive semidefinite matrices and a positive result}

In this section we first show that $AB \sim BA$ when $A$ is positive semidefinite and $B$ is normal.  Then we obtain a generalization by noting that the same proof works with significantly weaker conditions on $A$ and $B$.

\begin{theorem} \label{T:pos}
Let $A,B \in \Mn$, where $A$ is positive semidefinite and $B$ is normal.  Then $AB \sim BA$.
\end{theorem}

\begin{proof}
Because $B$ is normal and thus diagonalizable, after simultaneous unitary similarity we may assume that $$A = \left[ \begin{matrix} A_{11} & A_{12} \\ A_{12}^* & A_{22} \end{matrix} \right], \qquad B = \left[ \begin{matrix} C & 0 \\ 0 & 0 \end{matrix} \right],$$
where $C$ is an invertible diagonal matrix in $\mathbb{M}_r$ for some $0 \leq r \leq n$.

We claim that $A_{12} = A_{11} X$ for some $X \in \mathbb{M}_{r,n-r}$ (this is known, but we include the argument for discussion purposes below).  Suppose that $v \in \ker(A_{11}) \subseteq \CC^r$.  Then
$$\left[\begin{matrix} v \\ 0 \end{matrix} \right]^* A  \left[\begin{matrix} v \\ 0 \end{matrix} \right]  = 0,$$
which by positivity of $A$ implies that $\left[ \begin{matrix} v \\ 0 \end{matrix} \right] \in \ker(A)$, so that $v \in \ker(A_{12}^*)$ also.  The condition $\ker(A_{11}) \subseteq \ker(A_{12}^*)$ entails
\begin{equation} \label{E:range}
\range(A_{11}) = [\ker(A_{11})]^\perp \supseteq [\ker(A_{12}^*)]^\perp = \range(A_{12}),
\end{equation}
which implies the desired factorization: $A_{12} = A_{11}X$ for some $X \in \mathbb{M}_{r,n-r}$.

We have
$$A = \left[ \begin{matrix} A_{11} & A_{11}X \\ X^* A_{11} & A_{22}  \end{matrix} \right], \quad AB = \left[ \begin{matrix} A_{11}C & 0 \\ X^*A_{11} C & 0  \end{matrix} \right], \quad BA = \left[ \begin{matrix} CA_{11} & CA_{11}X \\ 0 & 0  \end{matrix} \right].$$

The matrix
$$S = \left[ \begin{matrix} C + XX^* & -X \\ -X^* & I  \end{matrix} \right] = \left[ \begin{matrix} I & -X \\ 0 & I  \end{matrix} \right] \left[ \begin{matrix} C  & 0 \\ 0 & I  \end{matrix} \right] \left[ \begin{matrix} I & 0 \\ -X^* & I  \end{matrix} \right]$$
is invertible and satisfies
\begin{align*}
S(AB) &= \left[ \begin{matrix} C + XX^* & -X \\ -X^* & I  \end{matrix} \right]  \left[ \begin{matrix} A_{11}C & 0 \\ X^*A_{11} C & 0  \end{matrix} \right] =  \left[ \begin{matrix} CA_{11}C & 0 \\ 0 & 0  \end{matrix} \right] \\ &= \left[ \begin{matrix} CA_{11} & CA_{11}X \\ 0 & 0  \end{matrix} \right]  \left[ \begin{matrix} C + XX^* & -X \\ -X^* & I  \end{matrix} \right] = (BA)S. \qedhere
\end{align*}
\end{proof}

Now let us isolate the properties of $A$ and $B$ that are essential to this proof.

Regarding $B$, the important point is unitary similarity to a matrix of the form $C \oplus 0$, where $C$ is invertible.  This is equivalent to requiring that $B$ have the same range as its adjoint; such matrices are called $EP$ or \textit{range Hermitian}.  (The name ``EP" originates in \cite[III.18]{S}, but its meaning as an abbreviation is not fully clear.)  Any normal matrix is EP.

Regarding $A$, we need the factorization $A_{12} = A_{11}X$; this is called the \textit{column inclusion property} for $A$.  In \cite[Observation 7.1.10 and preceding text]{HJ} it is shown that positive semidefinite matrices have the column inclusion property, essentially by the argument above.  The column inclusion property also holds under the weaker assumption that the real part of $A$ is positive semidefinite and has the same rank as $A$ (\cite[Observation 7.1.12]{HJ}).

This leads to the following generalization of Theorem \ref{T:pos}, proved in exactly the same way.  Note that neither $A$ nor $B$ is required to be normal.
 
\begin{theorem} \label{T:final}
Let $A,B \in \Mn$, where the real part of $A$ is positive semidefinite and has the same rank as $A$, and $B$ is $EP$.  Then $AB \sim BA$.
\end{theorem}
 
Theorems \ref{T:pos} and \ref{T:final} fail for infinite-dimensional operators, as demonstrated by the example in Remark \ref{T:inf}.  The reader may wonder where the proof goes wrong, as the range containment in \eqref{E:range} would still guarantee the factorization $A_{12} = A_{11} X$ by Douglas's theorem \cite{D}.  The issue is that ranges need not be closed, and in \eqref{E:range} we can only conclude that 
$$\overline{\range(A_{11})} \supseteq \overline{\range(A_{12})},$$
which does not suffice for the factorization.

\end{document}